\documentclass[12pt, reqno]{amsart}
\usepackage{amsmath}
\usepackage{amsfonts}
\usepackage{amssymb}
\usepackage[english]{babel}
\usepackage{color}
\usepackage{comment}
\usepackage{calrsfs}
\usepackage{accents}
\usepackage[utf8x]{inputenc}
\usepackage{hyperref}
\usepackage{enumitem}
\usepackage{breqn}
\usepackage{bbding}

\theoremstyle{plain}
\theoremstyle{definition}
\newtheorem {theorem}{Theorem}[section]
\newtheorem {lemma}[theorem]{Lemma}

\newtheorem {proposition} [theorem]{Proposition}
\newtheorem{definition}[theorem]{Definition}
\newtheorem{remark}[theorem]{Remark}

\theoremstyle{remark}

\newcommand{\R}{\mathbb R}
\newcommand{\N}{\mathbb N} 
\newcommand{\Lie}{\text{Lie}}
\newcommand{\Tan}{\text{Tan}}

\newcommand{\g}{\mathfrak{g}}

\newcommand{\barint}
{\rule[.036in]{.12in}{.009in}\kern-.16in \displaystyle\int}
\newcommand\restr[2]{{
		\left.\kern-\nulldelimiterspace
		#1
		\vphantom{\big|}
		\right|_{#2}
}}

\DeclareMathOperator{\dimens}{dim}

\newcommand{\set}[1]{\left\{#1\right\}}
\newcommand{\pa}[1]{\left(#1\right)}

\newcommand{\norm}[1]{\left\|#1\right\|}

\renewcommand{\epsilon} {\varepsilon}
\renewcommand{\phi} {\varphi}
\renewcommand{\theta}{\vartheta}

\renewcommand{\d}{\delta}

\renewcommand{\bar}{\overline}

\newcommand{\ex}{\mathrm e}

\long\def\MSC#1\EndMSC{\def\arg{#1}\ifx\arg\empty\relax\else
	{\par\narrower\noindent
		{\small\it 2010 Mathematics Subject Classification.} \small #1\par}\fi}
\long\def\KEY#1\EndKEY{\def\arg{#1}\ifx\arg\empty\relax\else
	{\par\narrower\noindent
		{\small\it Keywords and Phrases.} \small #1\par}\fi}

\begin{document}

\author[R.~Monti]{Roberto Monti}
\email{monti@math.unipd.it}

\author[A.~Pigati]{Alessandro Pigati}
\email{alessandro.pigati@math.ethz.ch}

\author[D.~Vittone]{Davide Vittone}
\email{vittone@math.unipd.it}

\address[Monti and Vittone]
{Universit\`a di Padova, Dipartimento di Matematica ``T. Levi-Civita'',
	via Trieste 63, 35121 Padova, Italy}

\address[Pigati]{ETH Z\"urich, Department of Mathematics,
	R\"amistrasse 101, 8092 Z\"urich, Switzerland}
	
\thanks{R.~M. and D.~V. are supported by MIUR (Italy)  and University of Padova. D.~V. is supported by INdAM-GNAMPA Project 2017 ``Campi vettoriali, superfici e perimetri in geometrie singolari''.}

\title[On tangent cones to length minimizers in CC spaces]
{On tangent cones to length minimizers \\ in Carnot--Carath\'eodory spaces}

\begin{abstract} 
We give a detailed proof of some facts about the blow-up of
horizontal curves in Carnot--Carath\'eodory spaces.
\end{abstract}
\maketitle

\section{Introduction}

We give a detailed proof of some facts about the blow-up of
horizontal curves in Carnot--Carath\'eodory spaces.
These results are crucially used in \cite{HL,LM,MPV}. The proof of
a fraction of these results was already sketched, in a special case, in \cite[Section 3.2]{R}.

Let $M$ be a connected
$n$-dimensional $C^\infty$-smooth manifold and  $\mathcal X=\{
X_1,\ldots,X_r\}$, $r\geq 2$,  a system of $C^\infty$-smooth vector fields on
$M$ that are pointwise linearly independent 
and satisfy the H\"ormander condition introduced below.
We call the pair $(M,\mathcal X)$ a \emph{Carnot--Carath\'eodory (CC) structure}.
Given an interval   $I\subseteq\R$, a Lipschitz curve $\gamma:I\to M$ is
said to be \emph{horizontal}
if there exist  functions $h_1,\ldots, h_r \in L^\infty(I)$ such that for
a.e.~$t\in I$ we have
\begin{equation}\label{horiz}
\dot \gamma (t) = \sum_{i=1}^r h_i(t) X_i (\gamma(t)).
\end{equation}
The function $h\in L^\infty(I;\R^r)$ is called the \emph{control} of $\gamma$. Letting $|h|:=(h_1^2+\ldots+h_r^2)^{1/2}$, the length of $\gamma$ is then
defined as
\[
L(\gamma):=\int_I |h(t)|\,dt.
\]
Since $M$ is connected, by the Chow--Rashevsky theorem (see e.g. \cite{chow,rashevski,cara}) for any pair of
points $x,y \in M$ there exists a horizontal curve
joining $x$ to $y$. We can therefore define a distance function $d: M\times
M\to[0,\infty)$ letting
\begin{equation}\label{diddi}
d(x,y) := \inf \big\{ L(\gamma) \mid \gamma: [0,T]\to M \textrm{ horizontal with
	$\gamma(0) = x$ and $\gamma(T)=y$}\big\}.
\end{equation}
The resulting metric space $(M,d)$ is a \emph{Carnot--Carath\'eodory  space}.
Since our analysis is local, our results apply in particular to \emph{sub-Riemannian manifolds}
$(M, \mathcal D, g)$, where
$\mathcal D \subset TM$ is a completely non-integrable distribution and $g$ is a
smooth metric on $\mathcal D$.

If the closure of any ball in $(M,d)$ is compact, then
the infimum in \eqref{diddi} is a minimum, i.e., any pair of points can be
connected by a length-minimizing curve. 
A horizontal curve $\gamma:[0,T]\to M$ is a \emph{length minimizer} if $L(\gamma) =
d(\gamma(0),\gamma(T))$.

The main contents of the paper are the following:
\begin{itemize}[leftmargin=*]

\item[(i)]  we define a tangent
Carnot--Carathéodory structure $(M^\infty,\mathcal{X}^\infty)$ at any point of
$M$, using exponential coordinates of the first kind, see Section \ref{nilpsec};

\item[(ii)] in Section \ref{tre},
  we
define the tangent
cone for a horizontal curve, at a given time, as the set of all possible
blow-ups in $(M^\infty,\mathcal{X}^\infty)$ of the curve, and we show that this
cone is always
nonempty, see Proposition \ref{TC};

\item[(iii)] 
we show that, if the curve has a
right derivative at the given time, the (positive) tangent cone consists of a single half-line,
see Theorem \ref{1.5};

\item[(iv)] 
if the curve is a length minimizer, in Theorem \ref{LM=LM} 
we show that all the
blow-ups are length
minimizers  in $(M^\infty,\mathcal{X}^\infty)$, as well;

\item[(v)] 
in Section \ref{4}, 
we show that a
tangent Carnot--Carathéodory structure can   be lifted to a free Carnot
group, in a way that preserves length minimizers. 
\end{itemize}

\section{Nilpotent approximation: definition of a tangent structure}\label{nilpsec}

In this section we introduce some basic notions about Carnot--Carath\'eodory
spaces.  Then we describe the
structure of a specific frame of vector fields $Y_1,\ldots,Y_n$ (constructed
below)  in exponential coordinates, see Theorem  \ref{U}. We also prove a lemma
describing the infinitesimal
behaviour of the Carnot--Carath\'eodory distance $d$ near $0$, with 
respect to suitable anisotropic dilations, see Lemma \ref{fare}.

We denote by $\Lie (X_1,\ldots,X_r)$ the real Lie algebra generated  by
$X_1,\ldots,X_r$ through iterated commutators.
The evaluation of this Lie algebra at a point $x\in M$ is a vector subspace of the
tangent space $T_xM$.
If, for any $x\in M$, we have
\begin{equation*}\label{hor}
\Lie (X_1,\ldots,X_r) (x) = T_x M,
\end{equation*}
we say that the system $\mathcal X = \{ X_1,\ldots,X_r\}$ satisfies the
\emph{H\"ormander condition} and we call the pair $(M, \mathcal X)$ a
\emph{Carnot--Carath\'eodory (CC) structure}.

Given a point $x_0\in M$, let $\phi \in
C^\infty(U;\R^n)$
be a chart such that $U$ is an open neighborhood of $x_0$ and $\phi(x_0)=0$.
Then $V:=\phi(U)$ is an open neighborhood of $0\in\R^n$ and the system of vector fields $Y_ i := \phi_* X_i$, with $i=1,\ldots,r$, still satisfies
the H\"ormander condition in $V$.

For a multi-index $J = (j_1,\ldots,j_k)$ with $k\ge 1$ and $j_1,\ldots,j_k \in
\{1,\ldots,r\}$, define the iterated commutator
\[ 
Y_J:=[Y_{j_1},\ldots,Y_{j_{k-1}},Y_{j_k}] 
\]
where, here and in the following, for given vector fields $V_1,\ldots, V_q$ we
use
the short notation $[V_1,\ldots ,V_q]$ to denote the commutator
$[V_1,[\cdots,[V_{q-1},V_q]\cdots]]$. 
We say that $Y_J$ is a commutator of \emph{length} $\ell(J) := k$ and we denote by $L^j$ the linear span of $\set{Y_J(0)\mid \ell(J)\le j}$,
so that
\[
\{ 0\}=L^0\subseteq L^1\subseteq\cdots\subseteq L^s=\R^n 
\]
for some minimal $s\ge 1$.
We select multi-indices $J_1=(1),\dots,J_r=(r),J_{r+1},\dots,J_n$ such that, for each $1\le j\le s$,
\[ \ell(J_{\dimens L^{(j-1)}+1})=\dots=\ell(J_{\dimens L^j})=j \]
and such that,
setting $Y_i:=Y_{J_i}$, the vectors $Y_1(0),\dots,Y_{\dimens L^j}(0)$ form a basis of
$L^j$. 
In particular, we have $\dimens L^1 =r$.

Possibly composing $\phi$ with a diffeomorphism (and shrinking $U$ and $V$), we can assume that $V$ is convex, that for any point $x=(x_1,\ldots,x_n) \in V$ we have
\begin{equation}\label{EXP}
x = \exp\Big(\sum_{i=1}^n x_i Y_i\Big)(0)
\end{equation}
and that $Y_1,\dots,Y_n$ are linearly independent on $V$.
Such coordinates $(x_1,\ldots,x_n)$ are called \emph{exponential coordinates of the first kind}
associated
with the frame $Y_1,\ldots, Y_n$. 
To each coordinate $x_i$ we assign the weight $w_i := \ell(J_i)$
and we define the anisotropic dilations $\delta_\lambda:\R^n\to\R^n$
\begin{equation}\label{dil}
\delta_\lambda (x) := (\lambda^{w_1} x_1,\ldots, \lambda ^{w_n} x_n),\qquad
\lambda >0.
\end{equation}

\begin{definition}
	A function $f:\R^n\to\R$ is $\delta$-\emph{homogeneous of degree}
$w\in\N$ if $f(\delta_\lambda (x)) = \lambda^w f(x)$ for all $x\in\R^n$,
$\lambda >0$.
We will refer to such a $w$ as the $\delta$\emph{-degree} of $f$.
\end{definition}

We will frequently use the anisotropic (pseudo-)norm 
\begin{equation}\label{eq:norma}
\| x\| \, := \sum_{i=1}^n |x_i|^{1/ w_i}, \qquad x\in\R^n.
\end{equation} The norm function, $x\mapsto\|x\|$, is
$\delta$-homogeneous of degree $1$.

We recall two facts about the exponential map, which are discussed
e.g.~in \cite[pp.~141--147]{NSW}.  First, for any $\psi\in
C^\infty(V)$, we have the  Taylor expansion 
\begin{equation}\label{eq:expexp} 
\psi\Big(\exp\Big(\sum_{i=1}^n s_i Y_i\Big)(0)\Big)\sim \pa{\ex^{\sum_i
s_i Y_i}\psi}(0)
\end{equation}
where
\begin{itemize}
\item the left-hand side is a function of $s\in\R^n$ near $0$;
\item the right-hand side is a shorthand for the formal series
\[ 
\sum_{k=0}^\infty\frac{1}{k!}((s_1 Y_1+\dots+s_n
Y_n)^k\psi)(0)=\sum_{k=0}^\infty\frac{1}{k!}\sum_{i_1,\dots,i_k\in\{1,\dots,n\}}
s_{i_1}\cdots s_{i_k}(Y_{i_1}\cdots Y_{i_k}\psi)(0);
\]
\item given a smooth function $f(x)$ and a formal power series $S(x)$, we define
the relation  $f(x)\sim S(x)$ if  the formal Taylor series of
$f(x)$  at 0  is $S(x)$.
\end{itemize}
 
Second, letting $S:=\sum_{i=1}^n s_i Y_i$ and $T:=\sum_{i=1}^n t_i Y_i$, the
following formal Taylor expansions hold as well: 
\begin{equation}
\label{eq:bchcampi}
\psi\Big(\exp(S)\circ\exp(T)(0)\Big)\sim\pa{\ex^{T} \ex^{S}\psi}(0)=(\ex^{P(T,
S)}\psi)(0),
\end{equation}
where
\begin{equation}\label{eq:serieP}
 P(T,S):=\sum_{p=1}^\infty\frac{(-1)^{p+1}}{p}\sum_{k_i+\ell_i\ge
1}\frac{[T^{k_1},S^{\ell_1},\ldots, T^{k_p},S^{\ell_p}]}{k_1!\cdots
k_p!\ell_1!\cdots\ell_p!(k_1+\ell_1+\dots+k_p+\ell_p)}.
\end{equation}
Above, the notation $T^k$ stands for  $T,\dots,T$, $k$
times.

\begin{remark}\label{rmkalgebra}
	The formal power series identity $\ex^T  \ex^S=\ex^{P(T,S)}$ is a
purely algebraic fact which holds in any (noncommutative, graded, complete)
associative real algebra, see e.g.~\cite[Sec. X.2]{Hoc}: this principle will be
used in the proofs of Theorem \ref{U} and Lemma \ref{fare}. 
\end{remark}

In the case of 
exponential coordinates {\em of the second kind}, the following theorem is
proved in  \cite{Her}.

\begin{theorem} \label{U}
	The vector
	fields $Y_1,\ldots,Y_n$ are of the
	form
	\begin{equation} \label{formo}
	Y_i (x) = \sum_{j=i}^ n a_{ij}(x) \frac{\partial}{\partial x_j}, \qquad 
	x\in V,
	%=
	%\frac{\partial}{\partial x_i} + \sum_{j>i} ^n a_{ij}(x)
	%\frac{\partial}{\partial x_j},
	\quad i=1,\ldots,n,
	\end{equation}
	where $a_{ij} \in C^\infty(V)$ are functions such that
	$a_{ij} =p_{ij}+ r_{ij}$ and:
	\begin{itemize}
		\item[(i)] for $w_j\ge w_i$, $p_{ij}$ are $\delta$-homogeneous polynomials in $\R^n$ of 
		degree
		$w_j-w_i$;
		
		\item[(ii)]  for $w_j\le w_i$, $p_{ij} = \delta_{ij}$ (in particular, $p_{ij}=0$ for $w_j<w_i$);
		
		\item[(iii)]   $r_{ij} \in C^\infty(V)$ satisfy  $r_{ij}(0)=0$;
		
		\item[(iv)]
		for $w_j\ge w_i$, $r_{ij}(x) =
		o(\|x\|
		^{w_j-w_i})$ as $x\to0$.
	\end{itemize}
	
\end{theorem}

\begin{proof} Suppose for a moment that  
\begin{equation}\label{eq:formouno} 
	a_{ij}(x)=O(\|x\|^{w_j-w_i}),\qquad i,j=1,\dots,n,\  w_j\ge w_i. 
\end{equation} 
Let  $p_{ij}$ be the sum of all monomials of $\delta$-degree 
$w_j-w_i$ in the Taylor expansion of $a_{ij}$, with the convention that $p_{ij}=0$ if $w_j<w_i$.
Statements (i) and (iv) then hold by construction, while (ii) and (iii) follow
from $a_{ij}(0)=\delta_{ij}$, which is a consequence of \eqref{EXP}.

Let us show \eqref{eq:formouno}. We pullback the identity $Y_i(x)=\sum_j a_{ij}(x)\frac{\partial}{\partial x_j}$ to the origin
using the map $\exp(-X)$ (locally defined near $x$), where $X:=\sum_k x_kY_k$, for a fixed $x\in V$. We  have
	\begin{equation}\label{eq:formopull} 
	\exp(-X)_* (Y_i(x))=\sum_j a_{ij}(x)\exp(-X)_*\Big(\frac{\partial}{\partial x_j}(x)\Big),
	\end{equation}
where the sum ranges from 1 to $n$. The above equation reads
	\[ 
	\sum_\ell b_{i\ell}(x)Y_\ell(0)=\sum_{j,\ell}a_{ij}(x)c_{j\ell}(x)Y_\ell(0) 
	\]
for suitable smooth coefficients $b_{i\ell}(x),c_{j\ell}(x)$. We claim that 
	\[
	b_{i\ell}(x)=O(\|x\|^{w_\ell-w_i}),\quad c_{j\ell}(x)=O(\|x\|^{w_\ell-w_j}),\quad\text{and}\quad c_{j\ell}(0)=\delta_{j\ell}.
	\]
Then, defining  $A:=(a_{ij})$, $B:=(b_{i\ell})$ and $C:=1-(c_{j\ell})$ ($1$ denoting the identity matrix), we obtain three $n\times n$ matrices  satisfying $B(x)=A(x)(1-C(x))$ and $C(0)=0$. In particular, $1-C(x)$ is invertible for $x$ close to 0 and $(1-C(x))^{-1}=\sum_{p=0}^\infty C(x)^p$. This gives 
	\[ 
	A(x)=\sum_{p=0}^s B(x) C(x)^p+o(|x|^s)=\sum_{p=0}^s B(x) C(x)^p+o(\|x\|^s) 
	\]
for any $s\in \N$, and \eqref{eq:formouno}   easily follows.
	
The proof of $c_{j\ell}(0)=\delta_{j\ell}$
 follows from the definition of $c_{j\ell}$ and from 
  $\frac{\partial}{\partial x_j}=Y_j(0)$,
  which in turn comes from \eqref{EXP}, as already observed.

	We prove the claim  $b_{i\ell}(x)=O(\|x\|^{w_\ell-w_i})$. By \eqref{EXP}, the left-hand side of \eqref{eq:formopull} satisfies
	\[ 
	\exp(-X)_* (Y_i(x)) = \frac{d}{dt}\exp(-X)\circ\exp(tY_i)\circ\exp(X)(0)\Big|_{t=0} .
	\]
Using \eqref{eq:bchcampi} and Remark \ref{rmkalgebra}, for any smooth $\psi$ we obtain
	\[ 
	\psi\big (\exp(-X)\circ\exp(tY_i)\circ\exp(X)(0)\big) \sim \ex^{P(P(X,tY_i),-X)}\psi(0),
	\]
the left-hand side being interpreted as a function of $(x,t)$.
We now differentiate this identity at $t=0$. 
Since $W(t) :=  P\big(P(X,tY_i),-X\big)  $ vanishes at $t=0$, one has $\frac{d}{dt}(\ex^{W(t)}\psi)(0)\Big|_{t=0}=\frac{d}{dt}(W(t)\psi)(0)\Big|_{t=0}$ 
and, letting $\psi$ range among the coordinate functions, 
we deduce that any finite-order expansion in $x$ of $\exp(-X)_* (Y_i(x))$ 
is a linear combination of terms
of the form
	\[ 
	x_{i_1}\cdots x_{i_p}[Y_{i_1},\ldots, Y_{i_m},Y_i, Y_{i_{m+1}},\ldots, Y_{i_p}](0)
	\]
where $p\geq 1$ and $0\leq m\leq p$. 	By Jacobi's identity, the
iterated commutator $[Y_{i_1},\ldots, Y_{i_m},Y_i, Y_{i_{m+1}},\ldots
,Y_{i_p}](0)$ is a linear combination of the vectors $Y_J(0)$ with $\ell(J)=\bar
w:=\sum_{q=1}^p w_{i_q}+w_i$ and so, by construction, it is a linear combination
of the vectors $Y_\ell(0)$ with $w_\ell\le\bar w$. Hence, letting
$w_\alpha:=\sum_{q=1}^n\alpha_q w_q$ for all $\alpha\in\N^n$, we have
	\[ 
	\exp (-X)_* (Y_i(x) )\sim\sum_\ell\sum_{\alpha:w_\alpha\ge w_\ell-w_i}d_{\alpha i\ell}x^\alpha Y_\ell(0) ,
	\]
	for suitable coefficients $d_{\alpha i\ell}\in\R$. This gives the required estimate.
	
The proof of $c_{j\ell}(x)=O(\|x\|^{w_\ell-w_j})$ is analogous to the preceding argument, once we observe that
	\[ 
	\exp(-X)_*\Big(\frac{\partial}{\partial x_j}(x)\Big)
	=\frac{d}{dt}\exp (-X)\circ\exp (X+t Y_j )(0)\Big|_{t=0}. 
	\]
	We can omit the details.
 \end{proof}

\begin{lemma} \label{fare}
	For any compact set $K\subset\R^n$ and any $\epsilon>0$ there exist $\delta>0$ and
	$\bar\lambda >0$ such that
	$\lambda  d(\delta_{1/\lambda } (x) , \delta_{1/\lambda  } (y))<
	\epsilon$ for all $x,y\in K$ with $|x-y|<\delta$ and all $\lambda \geq \bar \lambda $.
\end{lemma}

\begin{proof}
	Let $\psi\in C^\infty(V)$ be an arbitrary smooth function. Using \eqref{eq:expexp} and Remark \ref{rmkalgebra}, we have the following identity of formal power series in $(s,t)\in \R^n\times\R^n$: letting   $S:=\sum_{i=1}^n s_iY_i$ and $T:=\sum_{i=1}^n t_iY_i$,
	\begin{equation}\label{eq:seiex}
		\psi(\exp(S)(0))\sim(\ex^{S}\psi)(0)=(\ex^{T}\ex^{-T}\ex^{S}\psi)(0)=(\ex^{T}\ex^{P(-T,S)}\psi)(0). 
	\end{equation}
The truncation $P_N(-T,S) $ of the series $P(-T,S)$ up to $\delta$-degree $N:=w_n$ is 
	\begin{equation}\label{eq:truncexpl} P_N(-T,S) = \sum_{1\le\ell(J)\le N}q_J(s,t)Y_J, 
	\end{equation}
	where the sum is over all $J$ such that 
	$1\le\ell(J)\le N$ and 
	$q_J$ is a homogeneous polynomial with $\delta$-degree $\ell(J)$,
	i.e., $q_J(\delta_\lambda s,\delta_\lambda t)=\lambda^{\ell(J)}q_J(s,t)$. 
This follows from the fact that any iterated commutator $[Y_{i_1},\dots, Y_{i_k}]$ is a constant linear combination of the vector fields $Y_J$'s with $\ell(J)=\sum_{j=1}^k w_{i_j}$ (which in turn is a consequence of Jacobi's identity).
	
	Moreover, using \eqref{eq:truncexpl} and applying \eqref{eq:bchcampi} with the vector fields $Y_J$ in place of $Y_1,\dots,Y_n$, we have the following formal Taylor expansion in $(s,t)$ at $0\in\R^{2n}$
	\[ 
	\psi\big(\exp (P_N(-T,S) )  \circ\exp(T)(0)\big)\sim
	\Big(\ex^{T} \ex^{P_N(-T,S) }\psi\Big)(0), 
	\]
	which, by \eqref{eq:seiex}, coincides with the one of $\psi(\exp(S)(0))$ up to $\delta$-degree $N$. Since this holds for any $\psi$, we deduce (for instance letting $\psi$ range among the coordinate functions) that
	\[ 
	\exp (S)(0)=\exp (P_N(-T,S))  \circ \exp (T)(0)+o(|s|^N+|t|^N), 
	\]
	which by \eqref{EXP} gives
	\[ 
	s=\exp(P_N(-T,S)) (t)+o(|s|^N+|t|^N)=:f(s,t)+o(|s|^N+|t|^N).
	\]
	
	Now let $s=\delta_{1/\lambda}(x)$ and $t=\delta_{1/\lambda}(y)$ with $x,y\in K$.
	Since
	\[ 
	q_J(s,t)=\lambda^{-\ell(J)}q_J(x,y), 
	\]
	by \cite[Theorem~4]{NSW} we get
	\[ 
	d(t,f(s,t))\le C\sum_{1\le\ell(J)\le N}|q_J(s,t)|^{1/\ell(J)}
	=C\lambda^{-1}\sum_{ 1\le\ell(J)\le N}|q_J(x,y)|^{1/\ell(J)}, 
	\]
	while, by \cite[Lemma~2.20(b)]{NSW},
	\[ 
	d(s,f(s,t))=O(|s-f(s,t)|^{1/w_n})=o(|s|+|t|)=o(\lambda^{-1}), 
	\]
	provided  $\lambda$ is sufficiently large.
	Thus, by the triangle inequality,
	\[ 
	\lambda d(\delta_{1/\lambda}(x),\delta_{1/\lambda}(y))=\lambda d(s,t)\le C\sum_{1\le\ell(J)\le N}|q_J(x,y)|^{1/\ell(J)}+\frac{\epsilon}{2} 
	\]
	for all  $\lambda\ge\bar\lambda$, for a suitably large $\bar\lambda>0$. Finally, since $P_N(S,-S)=0$, we can assume that $q_J$ vanishes on the diagonal of $K\times K$ (possibly replacing $q_J(s,t)$ with $q_J(s,t)-q_J(s,s)$). Hence, by compactness of $K$, we also have
	\[ 
	C\sum_{1\le\ell(J)\le N}|q_J(x,y)|^{1/\ell(J)}<\frac{\epsilon}{2} 
	\]
	whenever $x,y\in K$ are such that  $|x-y|<\delta$, for a suitably small $\delta>0$.
\end{proof}

We now introduce the vector fields $Y_1^\infty, \ldots,Y_r^\infty$ in $\R^n$ defined by
	\[
	Y_i^\infty(x) := \sum_{j=1}^n  p_{ij}(x) \frac{\partial}{\partial x_j},
	\]
	and we let $\mathcal X ^{\infty} = \{Y_1^\infty,
\ldots, Y_r^\infty\}$.
 The vector fields $Y_1^\infty, \ldots,Y_r^\infty$ are known as the {\em nilpotent approximation} of $Y_1,\ldots,Y_r$ at the point
$0$.
By Proposition \ref{nilpotenza} below, the pair $(\R^n, \mathcal X^\infty)$ 
is a Carnot--Carath\'eodory structure.
We set $M^\infty:=\R^n$ and we call $(M^\infty, \mathcal
X^\infty)$ a {\em tangent} Carnot--Carath\'eodory structure to $(M,\mathcal
X)$ at the point $x_0 \in M$.

\begin{proposition}\label{nilpotenza}
	The vector fields $Y_1^\infty, \ldots,Y_r^\infty$ are pointwise linearly independent and satisfy the H\"ormander condition in $\R^n$.
	Moreover, any iterated commutator $Y_J^\infty:=[Y_{j_1}^\infty,[\dots,[Y_{j_{k-1}}^\infty,Y_{j_k}^\infty]\dots]]$ of length $\ell(J)=k>s$ vanishes identically.
\end{proposition}

\begin{proof}
We claim that Theorem \ref{U} implies $Y_i^\infty=\lim_{\lambda\to\infty}\lambda^{-1}(\delta_\lambda)_*Y_i$, for all $i=1,\dots,r$, in the (local) $C^\infty$-topology (the vector field $\lambda^{-1}(\delta_\lambda)_*Y_i$ being defined on $\delta_\lambda(V)$). Indeed, since $Y_i(x)=Y_i^\infty(x)+\sum_j  r_{ij}(x)\frac{\partial}{\partial x_j}$, we have
	\[ 
	\lambda^{-1}((\delta_{\lambda})_*Y_i)(x)=Y_i^\infty(x)+\sum_{j=1}^n\lambda^{w_j-1} r_{ij}(\delta_{1/\lambda}(x))\frac{\partial}{\partial x_j} ,
	\]
because $\lambda^{-1}(\delta_\lambda)_*Y_i^\infty=Y_i^\infty$. By Theorem
\ref{U}, 
the monomials in the Taylor expansion of $  r_{ij}$ have
$\delta$-degree greater than $w_j-1$. Thus, for any $\alpha\in\N^n$,
	\[ 
	\frac{\partial^{|\alpha|}}{\partial x^\alpha}(\lambda^{w_j-1} r_{ij}(\delta_{1/\lambda}(x)))
	=\lambda^{w_j-1-w_\alpha}\frac{\partial^{|\alpha|} r_{ij}}{\partial x^\alpha}(\delta_{1/\lambda}(x)), 
	\]
	where $w_\alpha:=\sum_\ell\alpha_\ell w_\ell$. The monomials in the expansion of $\frac{\partial^{|\alpha|} r_{ij}}{\partial x^\alpha}$ have $\delta$-degree greater than $w_j-1-w_\alpha$, hence $\Big|\frac{\partial^{|\alpha|} r_{ij}}{\partial x^\alpha}(\delta_{1/\lambda}(x))\Big|=o(\lambda^{-(w_j-1-w_\alpha)})$ and the claim follows.
	
	In particular, we deduce that for any multi-index $J$
	\begin{equation}\label{eq:convcomm}
	Y_{J}^\infty=\lim_{\lambda\to\infty}\lambda^{-\ell(J)}(\delta_\lambda)_*Y_{J} ,
	\end{equation}
	in the local $C^\infty$ topology. 	Hence, defining the $n\times n$ matrix $D_{\lambda}:=$diag$[\lambda^{w_1},\dots,\lambda^{w_n}]$  
	and recalling that $\ell(J_p)=w_p$,  for all $p=1,\dots,n$ we have
	\[ 
	Y_{J_p}^\infty(x)=\lim_{\lambda\to\infty}\lambda^{-w_p}D_\lambda Y_{J_p}(\delta_{1/\lambda}(x)). 
	\]
	Now the first statement follows from
	\[ 
	\begin{split} 
	\det(Y_{J_1}^\infty,\dots,Y_{J_n}^\infty)(x)&=\lim_{\lambda\to\infty}\lambda^{-\sum_i w_i}\det(D_\lambda)\det(Y_{J_1},\dots,Y_{J_n})(\delta_{1/\lambda}(x)) \\
	&=\det(Y_{J_1},\dots,Y_{J_n})(0)=\det(Y_1,\dots,Y_n)(0), 
	\end{split} 
	\]
	which is a nonzero constant. This gives the first part of the statement.
	
In order to prove the last assertion, we use again the fact that $\lambda^{-1}(\delta_\lambda)_*Y_i^\infty=Y_i^\infty$ for $i=1,\dots,r$. For any $x\in\R^n$ and any $J$ with $\ell(J)>s=w_n$ we have, by \eqref{eq:convcomm},
	\[ 
	Y_J^\infty(x)=\lim_{\lambda\to\infty}\lambda^{-\ell(J)}
((\delta_\lambda)_*Y_J)(x)
=\lim_{\lambda\to\infty}\lambda^{-\ell(J)}D_\lambda Y_J(\delta_{1/\lambda}(x)).
	\]
	The right-hand side is bounded by
$\lambda^{s-\ell(J)}|Y_J(\delta_{1/\lambda}(x))|$ (if $\lambda\ge 1$), which
tends to $0$ as $\lambda\to\infty$. This shows that $Y_J^\infty=0$.
\end{proof}

\begin{remark}
	Setting $Y_i^\infty:=Y_{J_i}^\infty$ for $i=1,\dots,n$, the coordinate functions on $M^\infty=\R^n$ are exponential coordinates of the first kind for $(Y_1^\infty,\ldots, Y_n^\infty)$, namely
\begin{equation}\label{EXPbis}
x = \exp\Big(\sum_{i=1}^n x_i Y_i^\infty\Big)(0).
\end{equation}
for any $x\in \R^n$. This follows from the fact that, for $\lambda$ large enough (depending on $x$), we have $y:=\delta_{\lambda^{-1}}(x)\in V$ and, using \eqref{EXP} with $y$ in place of $x$,
\[ x=\delta_\lambda\Big(\exp\Big(\sum_i y_i Y_i\Big)(0)\Big)=\exp\Big(\sum_i x_i\lambda^{-w_i}(\delta_\lambda)_*Y_i\Big)(0)\to\exp\Big(\sum_i x_i Y_i^\infty\Big)(0) \]
as $\lambda\to\infty$, since \eqref{eq:convcomm} gives $\lambda^{-w_i}(\delta_\lambda)_*Y_i\to Y_i^\infty$ in the local $C^\infty$ topology.
\end{remark}

\section{The tangent cone to a horizontal curve}\label{tre}

Let $(M,\mathcal X)$ be a CC
structure and let  $\gamma: [-T,T]\to M$ be a horizontal curve. 
Given $t\in (-T,T)$,
let $\phi$ be a chart centered at $x_0 = \gamma(t)$, as in the previous section,
together with the dilations $\delta _\lambda$ and the tangent CC structure $(M^\infty,\mathcal X^\infty)$ introduced
above.

\begin{definition}\label{TAN}
	The {\em tangent cone} $\Tan(\gamma;t)$ to $\gamma$ at $t\in(-T,T)$ is the set
	of
	all
	horizontal
	curves $\kappa:\R\to M^\infty$ such that there exists
	an infinitesimal sequence $\eta_i\downarrow 0 $ satisfying,
	for any $\tau\in \R$,
	\[
	\lim_{i\to\infty} \delta_{1/\eta_i}\phi \big(\gamma(t+\eta_i\tau)\big)
	=\kappa(\tau),
	\] with uniform convergence on compact subsets of $\R$. 
\end{definition}

We remark that any limit curve as above is automatically $(M^\infty,\mathcal X^\infty)$-horizontal: see e.g. the proof of Theorem \ref{LM=LM}. 

The definition of $\Tan(\gamma; t)$ depends on the 
choice $Y_1,\ldots,Y_n$ of linearly independent iterated commutators.
When $\gamma: [0,T] \to M$, the tangent cones  $\Tan^+(\gamma;0)$ and 
$\Tan^-(\gamma;T)$
can be defined in a similar way:  $\Tan^+(\gamma;0)$ contains curves in
$M^\infty$ defined  on $[0,\infty)$, while $\Tan^-(\gamma;T)$
contains curves defined on $(-\infty,0]$.

When $M=M^\infty$ or $M=G$ is a Carnot group,
there is already a group of dilations on $M$ itself.
In such cases, when $\gamma(t)=0$, we define the tangent
cone $\Tan(\gamma;t)$
as the set of horizontal limit curves of the form
$\displaystyle \kappa(t)=\lim_{i\to\infty}\delta_{1/\eta_i}\gamma(t+\eta_i\tau)$.
 
The tangent cone is closed under uniform convergence of curves
on compact sets.

\begin{proposition} \label{TC}
	For any horizontal curve  $\gamma: [-T,T] \to M$  the tangent
	cone $\Tan(\gamma; t) $ is nonempty for any $t\in (-T,T)$. The same
	holds for $\Tan^+(\gamma;0)$ and 
	$\Tan^-(\gamma;T)$, for a horizontal curve $\gamma:[0,T]\to
	M$.
\end{proposition}

\begin{proof}
	We prove  that $\Tan^+(\gamma;0)\neq \emptyset$. The other cases  are
	analogous.  
	
	We use exponential coordinates of the first kind centered at $\gamma(0)$.
	By \eqref{horiz}, we have a.e.
	\[
	\dot\gamma =\sum_{i=1}^r h_i Y_i(\gamma) = \sum_{j=1}^n \sum_{i=1}^r h_i 
	a_{ij}(\gamma) \frac{\partial}{\partial x_j},
	\]
	where $h_i \in L^\infty([0,T])$ and $a_{ij}=p_{ij}+  r_{ij}$, as in Theorem \ref{U}.
	Letting $K:=\gamma([0,T])$,  we have  $|\dot\gamma(t)| \leq C $ for some constant
	depending on
	$  \| a_{ij}\|_{L^\infty(K)}$ and $\|h\|_{L^\infty}$. This implies that $|\gamma(t)|\leq C t$ for all
	$t\in[0,T]$.
	
	By induction on $k\ge 1$, we prove the following statement:
	for any $j$ satisfying $w_j \ge k$ we have $|\gamma_j(t)|\le Ct^k$. 
	The base case $k=1$ has already been treated. 
	Now assume that $w_j\ge k>1$ and that the statement is true for $1,\dots,k-1$.
	Since $r_{ij}$ is smooth, we have $r_{ij}=q_{ij,k}+r_{ij,k}$, where $q_{ij,k}$
	is a polynomial containing only terms with
	$\delta$-homogeneous degree at least $w_j-w_i+1=w_j$ and $|r_{ij,k}(x)|\le C|x|^{k-1}$ on
	$K$ (here $|x|$ denotes the usual Euclidean norm).
	
	Each monomial $c_\alpha x^\alpha$ of the polynomial $p_{ij}+q_{ij,k}$ has
	$\delta$-degree $w_\alpha\ge w_j-1$. If
	$\alpha_m=0$ whenever $w_m\ge k$, then we can estimate
	\[ 
	|\gamma(t)^\alpha|=\prod_{m:w_m\le k-1}|\gamma_m(t)|^{\alpha_m}\le Ct^{w_\alpha}
	\le Ct^{k-1},
	\]
	using the inductive hypothesis with $k$ replaced by $w_m\le k-1$. Otherwise,
	there exists some index $m$ with $w_m\ge k$ and $\alpha_m>0$, in which case
	\[ 
	|\gamma(t)^\alpha|\le C|\gamma_m(t)|\le Ct^{k-1},
	\]
	using the inductive hypothesis with $k$ replaced by $k-1$.
	Thus $|p_{ij}(\gamma(t))+q_{ij,k}(\gamma(t))|\le Ct^{k-1}$.
	Combining this with
	the estimate $|r_{ij,k}(\gamma(t))|\le Ct^{k-1}$, we obtain
	$|a_{ij}(\gamma(t))|\le Ct^{k-1}$. So we finally have
	\[ 
	  |\gamma_j(t)|\le\| h\|_{L^\infty}\sum_{i=1}^r\int_0^t|a_{ij}(\gamma(\tau))|\,d\tau\le Ct^k, 
	\]
	completing the inductive proof. Applying the above statement with $k=w_j$, we obtain
	\begin{equation}\label{lia}
	|\gamma_j(t)| \leq C t^{w_j},
	\end{equation}
	for a suitable constant $C$ depending only on $K$, $T$ and $\norm{h}_{L^\infty}$.
	
	Now we prove that $\Tan^+(\gamma;0)$ is nonempty.
	For $\eta>0$ consider the family of curves $\gamma^\eta(t):=
	\delta_{1/\eta}(\gamma(\eta t))$, defined for $t\in[0, T/\eta]$.
	The derivative of $\gamma^\eta$ is a.e.
	\[
	\dot\gamma^\eta(t) =   \sum_{j=1}^n \sum_{i=1}^r h_i(\eta t) \eta^{1- w_j} 
	a_{ij}(  \gamma(\eta t)) \frac{\partial}{\partial x_j},
	\]
	where, by Theorem \ref{U} and the estimates \eqref{lia}, we have
	\[ |a_{ij}(  \gamma(\eta t))| \leq C \|  \gamma(\eta t)\| ^{w_j-1} \leq C
	(\eta t)^{w_j-1}. \]
	This proves that the family of curves $(\gamma^\eta)_{\eta>0}$ is locally 
	Lipschitz equicontinuous.
	So it has a subsequence $(\gamma^{\eta_i})_i$ that is converging locally uniformly as $\eta_i\to0$
	to a curve $\kappa:[0,\infty) \to\R^n $.
\end{proof}

\begin{remark}\label{rem:potenzecoordinatecurve}
	The following result was obtained along the proof of Proposition \ref{TC}.
	Let $(M,\mathcal X)$ be a Carnot--Carath\'eodory structure. Using
	exponential coordinates of the first kind, we (locally) identify $M$ with
	$\R^n$ and we assign to the coordinate $x_j$ the weight $w_j$, as above. Given $T>0$ and $K$ compact,  there exists a positive constant $C=C(K,T)$ such that the following holds: for any  horizontal curve $\gamma:[0,T]\to K$ parametrized by arclength
	and such that $\gamma(0) = 0$, one has 
	\begin{equation} 
	\label{}
	|\gamma_j(t)|\leq C t^{w_j},\quad \text{for any $j=1,\ldots,n$ and
		$t\in[0,T]$}. 
	\end{equation}
	In Carnot groups, by homogeneity, the constant $C$ is independent of $K$ and $T$.
\end{remark}

\begin{definition}
	We say that $v\in\R^n$ is a \emph{right tangent vector} to a curve $\gamma:[0,T]\to\R^n$ 
	at $0$
	if
	\[
	\gamma(t)=tv+o(t),\quad\text{as }t\to 0^+.
	\]
	The definition of a \emph{left tangent vector} is analogous.
\end{definition}

The next result is stated in exponential coordinates of the first kind.

\begin{theorem} \label{1.5}
	Let $\gamma:[0,T]\to V$ be a horizontal curve parametrized by
	arclength, with $\gamma(0)=0$.
	If $\gamma$ has a right tangent vector $v\in\R^n$ at $0$, then:
	\begin{itemize}
		\item[(i)] $v_j=0$ for $j>r$ and $|v|\le 1$;
		\item[(ii)] $\Tan^+(\gamma;0) = \{\kappa\}$, where   $\kappa (t) = tv$ for		$t\in [0,\infty)$;
		\item[(iii)] $|v|=1$ if $\gamma$ is  also length minimizing.
	\end{itemize}
	A similar statement holds if $\gamma:[-T,0]\to V$ has a left tangent vector at $0$.
\end{theorem}

\begin{proof}
	(i) Since $Y_i(x)=\frac{\partial}{\partial x_i}+o(1)$ as $x\to 0$, we have
	\begin{equation}\label{eq:ovvia}
	\gamma_j(t) = \int_0^ t \sum_{i=1}^r h_i(s)\delta_{ij}\,ds + o(t).
	\end{equation}
	We deduce that $v_j=0$ for $j>r$ and
	\[ |v|=\lim_{t\to 0^+}\Big|\frac{\gamma(t)}{t}\Big|\le\lim_{t\to 0^+}\frac 1t\int_0^t|h(s)|\,ds=1. \]
	
	(ii) Since $\gamma_j(t)=v_jt+o(t)$ for $j\le r$, it suffices to show that
	\begin{equation}
	\label{CL}
	\gamma_j(t) =o(  t^{w_j}),\quad j>r.
	\end{equation}
	Up to a rotation of the vector fields $Y_1,\dots,Y_r$, which by \eqref{EXP} corresponds to a rotation of the first $r$ coordinates, we can assume that 
	$v_2=\ldots=v_r=0$. Notice that Theorem \ref{U} still applies in these new exponential coordinates.
	From \eqref{eq:ovvia} we get
	\begin{equation}
	\label{viva}
	\lim_{t\to0^+ } \frac 1t \int_0^ t h_i(s)\, ds =
	\left\{
	\begin{array}{ll}
	v_1 & i=1
	\\
	0 & i=2,\ldots,r.
	\end{array}
	\right.
	\end{equation}
	By Remark \ref{rem:potenzecoordinatecurve} we have $\|\gamma(t)\|=O(t)$.
	We now show \eqref{CL} by induction on $j\geq r+1$.
	
	Assume the claim holds for $ r+1,\ldots,j-1$.
	The coordinate $\gamma_{j}$, with $j>r$, is
	\[
	\gamma_{j}(t) = \sum_{i=1}^r \int_0^t h_i(s) a_{ij}(\gamma(s))\, ds
	=  \int_0^t h_1(s) a_{1j}(\gamma(s))\, ds
	+\sum_{i=2}^r \int_0^t h_i(s) a_{ij}(\gamma(s))\, ds.
	\]
	By Theorem \ref{U}, $a_{ij}=p_{ij}+r_{ij}$ with $r_{ij}(x) = o(\| x\|^{w_j-1})$, so we deduce that
	\[
	a_{ij}(\gamma(s)) = p_{ij}(\gamma(s))+ r_{ij}(\gamma(s))
	=p_{ij}(\gamma(s))+ o(s^{w_j-1}),\quad i=1,\dots,r.
	\]
	From \eqref{EXP} we deduce that for $i=1,\ldots,r$ we have $Y_i(0,\dots,x_i,\dots,0)=\frac{\partial}{\partial x_i}$, hence
	\begin{equation}\label{QW}
	a_{ij} (0,\ldots, x_i,\ldots,0) = 0,\quad j>r.
	\end{equation}
	The polynomial  $p_{ij}(x)$ is $\delta$-homogeneous of degree $w_j-w_i =w_j-1$
	and so it contains no variable $x_k$ with $k\geq j$.
	Condition \eqref{QW}  implies that $p_{ij}(x)$
	does not
	contain the monomial $x_i^{w_j-1}$, either. Thus, when $i=1$  each
	monomial in  $p_{1j}(x)$ contains at least one of the variables 
	$x_2,\ldots, x_{j-1}$.
	By the inductive assumption, it follows that  $p_{1j}(\gamma(s)) = 
	o(s^{w_j-1})$, and thus $a_{1j}(\gamma(s)) =
	o(s^{w_j-1})$. This implies that
	\[
	\int_0^t h_1(s) a_{1j}(\gamma(s))\, ds =o(t^{w_j}) .
	\]
	
	Now we consider the case  $i=2,\ldots,r$.
	Letting  $p_{ij} =c_{ij} x_1^{w_j-1}  + \widehat p_{ij}$ with $c_{ij}\in 
	\R$
	and  $\widehat{a}_{ij} := \widehat{p}_{ij} +r_{ij}$, we have
	$\widehat a_{ij} (\gamma(s)) = o (s^{w_j-1})$ as in the previous case and thus
	\[
	\int_0^t h_i(s) \widehat a_{ij}(\gamma(s))\,
	ds =o( t^{w_j}) .
	\]
	We claim that, for $i=2,\ldots,m$, we also have
	\[
	\int_0^t h_i(s)  \gamma_1(s) ^{w_j-1}\,
	ds =o(t^{w_j}).
	\]
	Indeed, since $v_i=0$ we have $H_i(s):=\int_0^s h_i(s')\,ds'=o(s)$, so integration by parts gives
	\[ \begin{split} \int_0^t h_i(s)  \gamma_1(s)^{w_j-1}\,ds&=H_i(t)\gamma_1(t)^{w_j-1}-(w_j-1)\int_0^t H_i(s)\gamma_1(s)^{w_j-2}\dot\gamma_1(s)\,ds \\
	&=o(t^{w_j})+\int_0^t o(s^{w_j-1})\,ds
	=o(t^{w_j}). \end{split} \]
	This ends the proof of  \eqref{CL} and hence of (ii).
	
	(iii) By Theorem \ref{LM=LM} below, $\kappa$ is parametrized by arclength.
	But $(v_1,\dots,v_r)$ equals its (continuous) control $h(t)$ at $t=0$, so $|v|=1$.
\end{proof}

For $\lambda>0$, we define the  vector fields
$Y_1^\lambda ,\ldots,Y_r^\lambda $ in $\delta_\lambda(V)$ by
\[
Y_i^\lambda (x) := \lambda^{-1}((\delta_\lambda)_*Y_i)(x)=\sum_{j=1}^n \lambda ^{w_j - 1} a_{ij}  
(\d_{1/\lambda}(x)) \frac{\partial}{\partial x_j},\quad x\in\delta_\lambda(V).
\]

\noindent In the proof of Proposition \ref{nilpotenza} it was shown that
\begin{equation}\label{[1]} Y_i^\lambda\to Y_i^\infty \end{equation}
locally uniformly in $\R^n$ as $\lambda\to\infty$, together with all the derivatives.

We denote by $d^\lambda$ the Carnot--Carath\'eodory metric of
$(\delta_\lambda(V),\mathcal{X}^\lambda)$, with $\mathcal{X}^\lambda:=\{ Y_1^\lambda,\dots,Y_r^\lambda\}$. The distance function $d^\lambda$
is related to the distance function $d$ via the formula
\begin{equation}
\label{[7]}
d^\lambda (x,y) = \lambda d \big( 
\delta_{1/\lambda}(x),\delta_{1/\lambda}(y)\big),
\end{equation}
for all $x,y\in \delta_\lambda(V) $ and $\lambda>0$.
Indeed, let $\gamma:[0,1]\to V$ be a horizontal curve
\begin{equation}\label{gamma}
\gamma(t) = \gamma(0) + \int_0^t \sum_{i=1}^r h_i(s) Y_i(\gamma(s))\, 
ds,\quad t\in[0,1],
\end{equation}
and define the curve $\gamma^\lambda :[0,\lambda]\to\delta_\lambda(V)$
\begin{equation}
\label{quado}
\gamma^\lambda(t) := \delta_{\lambda} \gamma(t/\lambda),\quad t\in 
[0,\lambda].
\end{equation}
Then we have
\begin{equation}
\label{[qua]}
\gamma^\lambda(t) =
\gamma^\lambda(0) +\int_0^t \sum_{i=1}^r h_i(s/\lambda) 
Y_i^\lambda(\gamma^\lambda(s))\, ds,\quad t\in[0,\lambda],
\end{equation}
and therefore the length of $\gamma^\lambda$ is
\begin{equation}
\label{[6]}
L^\lambda(\gamma^\lambda) = \int_0^\lambda\big| h(s/\lambda)\big|\, ds = 
\lambda \int_0^1 |h(s)|\,ds = \lambda L(\gamma).
\end{equation}

If $\gamma$ is length minimizing, then the curves in $\Tan(\gamma;t)$
are also locally  length minimizing. This is the content of the next theorem.

\begin{theorem} \label{LM=LM} Let $\gamma:[-T,T]\to M$ be a
	length-minimizing
	curve in $(M,\mathcal X)$, parametrized by arclength,
	and let $\gamma^\infty \in \Tan (\gamma;t_0)$ for some $t_0\in (-T,T)$.
	Then $\gamma^\infty$ is  horizontal, parametrized by arclength and, when restricted to any compact interval, it is length minimizing in the
	tangent Carnot--Carath\'eodory structure $(M^\infty, \mathcal X^\infty)$.
\end{theorem}

\begin{proof}
	We can assume $t_0=0$. We use exponential coordinates of the first kind centered at $\gamma(0)$. Given any $\bar T>0$,
	for some sequence 
	$\lambda_h\to\infty$ we have
	\begin{equation} \label{CU}
	\gamma^{\lambda_h}(t):=\delta_{\lambda_h}\gamma(t/\lambda_h)
	\to\gamma^\infty(t)\text{ in }L^\infty([-\bar T,\bar T]).
	\end{equation}
	Up to a subsequence, we can assume that the functions
	$h(t/\lambda_h)$ weakly converge in $L^2([-\bar T,\bar T]; \R^r)$
	to some $h^\infty \in L^2([-\bar T,\bar T];\R^r)$ such that $|h^\infty|\leq 1$ 
	almost everywhere.
	Then, using \eqref{[qua]}, we have
	\[ \gamma^\infty(t)=\lim_{h\to\infty}\int_0^t\sum_{i=1}^r h_i(s/\lambda_h)Y_i^{\lambda_h}(\gamma^{\lambda_h}(s))\,ds=\int_0^t\sum_{i=1}^r h_i^\infty Y_i^\infty(\gamma^\infty(s))\,ds, \]
	so $\gamma^\infty$ is $(M^\infty,\mathcal{X}^\infty)$-horizontal and, denoting by $d^\infty$ the Carnot--Carathéodory distance  on $M^\infty$ induced by the family $\mathcal{X}^\infty$, its length  satisfies
	\begin{equation}\label{lalla}
	d^\infty(\gamma^\infty(-\bar T),\gamma^\infty(\bar T))\le L^\infty\Big(\restr{\gamma^\infty}{[-\bar T,\bar T]}\Big) = \int_{-\bar T}^{\bar T}|h^\infty| \,dt \leq 2\bar T.
	\end{equation}
	We will see that, in fact, the converse inequality
	$d^\infty(\gamma^\infty(-\bar T),\gamma^\infty(\bar T)) \ge  2\bar T$ holds as well, thus proving that $\gamma^\infty$ is length minimizing on $[-\bar T,\bar T]$ and parametrized by arclength (with  control $h^\infty$).
	
	Let $\kappa^\infty:[-\bar T,\bar T]\to\R^n$ be an $(M^\infty,\mathcal X^\infty)$-horizontal curve such that 
	$\kappa^\infty(\pm\bar T)=\gamma^\infty(\pm\bar T)$, with  control $k ^\infty\in L ^\infty([-\bar T,\bar T];\R^n)$.
	For all $h$ large enough, the ordinary differential equation
	\begin{equation}
	\label{INT}
	\dot\kappa^{\lambda_h} (t) = \sum_{i=1}^r k_i^\infty (t) 
	Y_i^{\lambda_h}(\kappa^{\lambda_h}(t))
	\end{equation}
	with initial condition $\kappa^{\lambda_h}(-\bar T)=\kappa^\infty(-\bar T)$ has a (unique) solution defined on $[-\bar T,\bar T]$.
	Indeed, let $K$ be a compact neighborhood of $\kappa^\infty([-\bar T,\bar T])$. For any $\epsilon>0$ we have $\|Y_i^{\lambda_h}-Y_i^\infty\|_{L^\infty(K)}\le\epsilon$ eventually. If $-\bar T\in I\subseteq[-\bar T,\bar T]$ is the maximal (compact) subinterval such that $\kappa^{\lambda_h}$ is defined on $I$ and $\kappa^{\lambda_h}(I)\subseteq K$, we have
	\[ |\dot\kappa^{\lambda_h}-\dot\kappa^\infty|\le C\epsilon+C\sum_i|Y_i^\infty(\kappa^{\lambda_h})-Y_i^\infty(\kappa^\infty)|\le C\epsilon+C|\kappa^{\lambda_h}-\kappa^\infty| \]
	on $I$, for some $C$ depending on $\|k^\infty\|_{L^\infty}$ and $\|\nabla Y_i^\infty\|_{L^\infty(K)}$. Hence, by Gronwall's inequality, $|\kappa^{\lambda_h}-\kappa^\infty|\le C\epsilon$ on $I$. If $\epsilon$ is small enough, we deduce that $\kappa^{\lambda_h}(\max I)$ belongs to the interior of $K$, so $I=[-\bar T,\bar T]$. Since $\epsilon$ was arbitrary, we also get
	\begin{equation} \label{pox}
	\lim_{h\to \infty}\kappa^{\lambda_h} (\pm\bar T) =\kappa^\infty 
	(\pm\bar T)=\gamma^\infty (\pm\bar T)=\lim_{h\to \infty}\gamma^{\lambda_h} (\pm\bar T).
	\end{equation}
	
	From the length minimality of $\gamma^{\lambda_h}$ in 
	$(\delta_{\lambda_h}(V),\mathcal{X}^{\lambda_h})$ it follows that
	\[ \begin{split}
	2\bar T = L^{\lambda_h} \Big(\restr{\gamma^{\lambda_h}}{[-\bar T,\bar T]}\Big)
	&\leq L^{\lambda_h}(\kappa^{\lambda_h}) +
	d^{\lambda_h}\Big( \kappa ^{\lambda_h}(-\bar T),\gamma^{\lambda_h}(-\bar T)\Big)
	+d^{\lambda_h}\Big( \kappa ^{\lambda_h}(\bar T),\gamma^{\lambda_h}(\bar T)\Big) \\
	&= \int_{-\bar T}^{\bar T} |k^\infty (t)|\, dt
	+\lambda_h d\Big(
	\delta_{1/{\lambda_h}} \kappa ^{\lambda_h}(-\bar T),\delta_{1/{\lambda_h}} \gamma ^{\lambda_h}(-\bar T)\Big) \\
	&\phantom{=}+\lambda_h d\Big(
	\delta_{1/{\lambda_h}} \kappa ^{\lambda_h}(\bar T),\delta_{1/{\lambda_h}} \gamma ^{\lambda_h}(\bar T)\Big).
	\end{split} \]
	By Lemma \ref{fare} and \eqref{pox}, we have
	\[
	\lim_{h\to\infty} \lambda_h d(
	\delta_{1/{\lambda_h}} \kappa ^{\lambda_h}(\pm\bar T),\delta_{1/{\lambda_h}} \gamma ^{\lambda_h}(\pm\bar T)) =0.
	\]
	Hence, $2\bar T\le\int_{-\bar T}^{\bar T}|k^\infty(t)|\,dt=L^\infty(\kappa^\infty)$.
	Since $\kappa^\infty$ was arbitrary, we conclude that $d^\infty(\gamma^\infty(-\bar T),\gamma^\infty(\bar T))\ge 2\bar T$.
\end{proof}

The following fact is a special case of the general principle according to which
the tangent to the tangent is (contained in the) tangent.

\begin{proposition}\label{blow-blow}
	Let $\gamma:[-T,T]\to M$ be a horizontal curve and $t\in(-T,T)$. If
	$\kappa
	\in \Tan(\gamma;t)$
	and $\widehat \kappa \in \Tan(\kappa;0)$, then $\widehat \kappa \in \Tan(\gamma;
	t)$.
\end{proposition}

\begin{proof}
	We can assume without loss of generality that $t=0$. We use exponential coordinates of the first kind centered at $\gamma(0)$. Let $N>0$ be fixed. Since $\widehat \kappa \in \Tan(\kappa;0)$, 
	there exists an infinitesimal
	sequence $\xi_k\downarrow0$ such that,
	for all $t\in [-N,N]$ and $k\in\N$, we have
	\[
	\|
	\widehat \kappa (t) - \delta_{1/\xi_k} \kappa(\xi_k t)\| \leq \frac {1}{2^k}.
	\]
	Since $ \kappa \in \Tan(\gamma;0)$, there exists an infinitesimal sequence
	$\eta_k\downarrow0$ such that,
	for all $t\in [-N,N]$ and $k\in\N$, we have
	\[
	\|
	\kappa (\xi_k t) - \delta_{1/\eta_k} \gamma(\eta_k\xi_k t)\| \leq \frac
	{\xi_k}{2^k}.
	\]
	It follows that for the infinitesimal sequence $\sigma_k := \xi_k\eta_k$ we have,
	for all $t\in[-N,N]$,
	\[
	\|
	\widehat \kappa (t) -  \delta_{1/\sigma_k} \kappa(\sigma_k t)\|
	\leq
	\|
	\widehat \kappa (t) - \delta_{1/\xi_k} \kappa(\xi_k t)\|+
	\| \delta_{1/\xi_k} \kappa(\xi_k t) - \delta_{1/\sigma_k} \gamma(\sigma_k
	t)\|\leq \frac{1}{2^{k-1}}.
	\]
	The thesis now follows by a diagonal argument.
\end{proof}

When $\gamma:[0,T]\to M$, 
there are analogous  versions of Propositions \ref{LM=LM} and
\ref{blow-blow}  for $\Tan^+(\gamma;0)$ and $\Tan^-(\gamma;T)$. 

\begin{proposition}\label{linechar}
	Let $\kappa:\R\to M^\infty$ be a horizontal curve in $(M^\infty,\mathcal X^\infty)$. 
	The following statements are equivalent: 
	\begin{itemize}
		%\item[(i)] the   control of $\kappa$ is a.e. constant and $\kappa(0)=0$;
		\item[(i)] there exist   $c_1,\ldots,c_r \in\R$ such that 
		$  \dot\kappa  =  \sum_{i=1}^rc_i Y_i^\infty(\kappa )$  
		and 
		$\kappa(0)=0$;
		\item[(ii)]   there exists $x_0\in M^\infty$ such that $\kappa(t) = \delta_t(
		x_0)$ (here $\delta_t$ is defined by \eqref{dil} also for $t<0$). 
	\end{itemize}

\end{proposition}

\begin{proof} 
%The implication  	(i)$\Rightarrow$(ii) is trivial. 
We prove (i)$\Rightarrow$(ii).
	 Since $(\delta_\lambda)_*Y_i^\infty=\lambda Y_i^\infty$ for $\lambda\neq 0$, the curve $\delta_\lambda\circ\kappa(\cdot/\lambda)$ satisfies the same differential equation, so $\delta_\lambda\circ\kappa(t/\lambda)=\kappa(t)$; choosing $\lambda=t$ we get $\kappa(t)=\delta_t(\kappa(1))$.
	 
We check 	(ii)$\Rightarrow$(i). Up to rescaling time, we can assume that
$\dot\kappa(1)$ exists and is a linear combination of
$Y_1^\infty(\kappa(1)),\dots,Y_r^\infty(\kappa(1))$,
so $\dot\kappa(1)=\sum_i\bar h_i Y_i^\infty(\kappa(1))$ for some $\bar h\in\R^r$. If $h$ is the   control of
$\kappa$, for a.e.~$s$ we have
	\[ \sum_{i=1}^r\bar h_i Y_i^\infty(\kappa(1))=\dot\kappa(1)=s\frac{d}{dt}\kappa(t/s)\Big|_{t=s}=s\frac{d}{dt}(\delta_{1/s}\circ\kappa(t))\Big|_{t=s}=\sum_{i=1}^r h_i(s)Y_i^\infty(\kappa(1)), \]
	again because $s(\delta_{1/s})_*Y_i^\infty=Y_i^\infty$. Since $Y_1^\infty,\dots,Y_r^\infty$ are pointwise linearly independent (see Proposition \ref{nilpotenza}), we get $h=\bar h$ a.e.
\end{proof}

\begin{definition}
	We say that a horizontal curve  $\kappa$ in $(M^\infty,\mathcal X^\infty)$ 
	is a {\em horizontal line}  (through $0$)  if 
	one of the conditions (i)--(ii) of Proposition \ref{linechar} holds.
\end{definition}

\noindent
The definition of \emph{positive and negative half-line} is similar,
the formulas above
being required to hold for $t\geq0$ and  $t\leq0$, respectively.

\begin{remark}\label{contreform}
Let us observe the following fact. Let  $\gamma:[-T,T]\to M$ be a length minimizer parametrized by arclength with control $h=(h_1,\dots,h_r)$ and let $t\in(-T,T)$ be fixed. Then, the tangent cone $\Tan(\gamma;t)$ contains a horizontal line $\kappa$ in $M^\infty$ if and only if there exist an infinitesimal sequence $\eta_i\downarrow 0$ and a constant unit vector $c\in S^{r-1}$ such that
	\[ 
	h(t+\eta_i\,\cdot)\to c\qquad\text{in }L^2_{loc}(\R).
	\]
As usual, an analogous version holds for $\Tan^+(\gamma;0)$ and $\Tan^-(\gamma;T)$ in case $\gamma$ is a length minimizer parametrized by arclength on the interval $[0,T]$.

Let us prove our claim; we can set $t=0$. Assume that there exists a sequence $\eta_i\downarrow 0$ such that the curves $\gamma^i(\tau):=\delta_{1/\eta_i}\varphi(\gamma(\eta_i\tau))$ converge locally uniformly to a horizontal line $\kappa$ in the tangent CC structure $(M^\infty,\mathcal X^\infty)$; we have
\[ 
\gamma^i(\tau)=\int_0^\tau \sum_{j=1}^r h_j(\eta_i s)Y_j^{1/\eta_i}(\gamma^i(s))\,ds. 
\]
Up to subsequences we have $h(\eta_i\,\cdot)\rightharpoonup h_\infty$ in $L^2_{loc}(\R)$, with $\|h_\infty\|_{L^\infty}\le 1$. Since $Y_j^{1/\eta_i}\to Y_j^\infty$ locally uniformly, we obtain
\[ 
\kappa(\tau)=\int_0^\tau\sum_{j=1}^r h_\infty(s)Y_j^\infty(\kappa(s))\,ds. 
\]
By Proposition \ref{LM=LM}, $\kappa$ is parametrized by arclength. So $|h_\infty|=1$ a.e. and, since $\kappa$ is a horizontal line, $h_\infty$ is constant. Finally, for any compact set $K\subset\R$, we trivially have $\|h(\eta_i\,\cdot)\|_{L^2(K)}\to\|h_\infty\|_{L^2(K)}$, which gives $h(\eta_i\,\cdot)\to h_\infty$ in $L^2(K)$. The reverse implication (if $h(t+\eta_i\,\cdot)\to c\text{ in }L^2_{loc}(\R)$, then $\Tan(\gamma;t)$ contains a horizontal line) follows a similar argument.
\end{remark}

\section{Lifting the tangent structure to a free Carnot group}
\label{4}

In this section we show how a tangent CC structure $(M^\infty,\mathcal{X}^\infty)$ 
can be lifted to a free Carnot group $F$, by means of a desingularization process. We also show that length minimizers in $M^\infty$ lift to length minimizers in $F$.

Let $(M^\infty,\mathcal{X}^\infty)$ be a tangent CC structure as in Section
\ref{nilpsec}. The Lie algebra $\g$ generated by
$\mathcal{X}^\infty=(Y_1^\infty,\dots,Y_r^\infty)$ is 
nilpotent because, by Proposition \ref{nilpotenza}, any iterated commutator of
length greater than $s$ vanishes. The identity
$(\delta_\lambda)_*Y_i^\infty=\lambda Y_i^\infty$ implies that
$(\delta_\lambda)_*X\to 0$ pointwise as $\lambda\to 0$, for any $X\in\mathfrak
g$. We deduce that the $j$-th component of $X$ is a polynomial function depending only on the previous
variables. It follows that the flow   $(x,t)\mapsto\exp(tX)(x)$ is
a polynomial function in $(x,t)\in M^\infty\times\R$ and $X$ is
therefore complete.

Let $\mathfrak f$ be the free Lie algebra of rank $r$ and step $s$, with generators $W_1,\dots,W_r$.
The connected, simply connected Lie group $F$ with Lie algebra $\mathfrak f$ can be constructed explicitly as follows:
we let $F:=\mathfrak f$ and we endow $F$ with the group operation $A\cdot
B:=P(A,B)$, 
where  
\begin{equation}\label{eq:pfree} 
P(A,B)=\sum_{p=1}^s\frac{(-1)^{p+1}}{p}\sum_{1\le k_i+\ell_i\le
s}\frac{[A^{k_1},B^{\ell_1},\ldots, A^{k_p},B^{\ell_p}]}{k_1!\cdots
k_p!\ell_1!\cdots\ell_p!\sum_i (k_i+\ell_i)}. 
\end{equation}
This is a finite truncation of the series in \eqref{eq:serieP}: the omitted terms vanish  by the nilpotency of $\mathfrak f$. One readily checks that $P(A,0)=P(0,A)=A$ and $P(A,-A)=P(-A,A)=0$,
while the associativity identity $P(P(A,B),C)=P(A,P(B,C))$ is shown in \cite[Sec. X.2]{Hoc} for free Lie algebras and can be deduced for $\mathfrak f$ by truncation. For any $A\in F$, $t\mapsto tA$ is a one-parameter subgroup. From this, it is straightforward to check that $\mathfrak f$ identifies with the Lie algebra of $F$, with $\exp:\mathfrak f\to F$ given by the identity map. In particular, $\exp:\mathfrak f\to F$ is a diffeomorphism and we have
\begin{equation}\label{eq:bchfree} \exp(A)\exp(B)=\exp(P(A,B)),\quad A,B\in\mathfrak f. \end{equation}

The group $F$ is a \emph{Carnot group}, which means that it is a connected,
simply connected and nilpotent Lie group whose Lie algebra $\mathcal f$ is stratified, i.e.,
it has an assigned decomposition
$ \mathfrak f= \mathfrak f_1\oplus\cdots\oplus \mathfrak f_s $
satisfying $[\mathfrak f_1,\mathfrak f_{i-1}]=\mathfrak f_{i}$ and $[\mathfrak f,\mathfrak f_s]=\{0\}$ (in this case $\mathfrak f_1$ is the linear span of $W_1,\dots,W_r$).
The group $F$ just constructed is called the \emph{free Carnot group of rank $r$ and step $s$}.

\begin{proposition}\label{generation}
	The group $F$ is generated by $\exp(\mathfrak f_1)$.
\end{proposition}

\begin{proof}
See \cite[Lemma 1.40]{FS}.
\end{proof}

By the nilpotency of $\g$, there exists  a unique
homomorphism $\psi:\mathfrak f\to\g$ such that $\psi(W_i )=Y_i^\infty\in\g$ for $i=1,\dots,r$.
The group $F$ acts on $M^\infty$ on the right. 
The action $M^\infty\times F\to M^\infty$ is given by 
$
(x,f)\mapsto  x\cdot f:=\exp(\psi(A))(x),
$
where  $f=\exp(A)$. In fact, by \eqref{eq:bchfree}, for any $f'=\exp(B)$ we have
\begin{equation}
\label{action}
x\cdot(ff')=\exp(P(\psi(A),\psi(B)))(x)=\exp(\psi(B))\circ\exp(\psi(A))(x)=(x\cdot f)\cdot f'.
\end{equation} 
The second equality is a consequence of the formula $\exp(P(tY,tX))(x)=\exp(tX)\circ\exp(tY)(x)$ for $X,Y\in\g$ (with $P$ given by \eqref{eq:pfree}), which holds since both sides are polynomial functions in $t$, with the same Taylor expansion (by \eqref{eq:bchcampi}).
We define the map
\[ \pi^\infty:F\to M^\infty,\quad \pi^\infty(f):=0\cdot f, \]
where the dot
stands for the right action of $F$ on $M^\infty$.

Let  $\mathcal W := \{  W_1,\ldots,W_r\}$ and extend $\mathcal W$ to a basis $W_1,\dots,W_N$ of $\mathfrak f$ adapted to the
stratification.
Via the exponential map $\exp:\mathfrak f\to F$, the one-parameter
group of
automorphisms of $\mathfrak f$ defined by $W_k \mapsto 
\lambda ^i W_k $ if and only if $W_k \in \mathfrak f_i$
induces a one-parameter group of automorphisms 
$(\widehat \delta _\lambda)_{\lambda>0}$ of $F$, called \emph{dilations}.

If $A\in \mathfrak f_1$, for
any $\lambda>0$ and $x\in M^\infty$ we have the identity 
\begin{equation}\label{pipox}
\begin{split}
\exp(\lambda \psi(A)) (\delta_\lambda (x)) &  
= \delta _\lambda \big( \exp(\psi(A)) (x)\big),
\end{split}
\end{equation} 
which follows from $(\delta_\lambda)_*\psi(A)=\lambda \psi(A)$. 

\begin{definition}
	We call the CC structure $(F,\mathcal W)$ the \emph{lifting} of $(M^\infty,\mathcal X^\infty)$
	with projection $\pi^\infty : F \to M^\infty$.
\end{definition}

\begin{proposition} \label{hel}
	The lifting $(F,\mathcal W)$   of
	$(M^\infty,\mathcal X^\infty)$ has the following properties:
	\begin{itemize}
		\item[(i)]
		for any $f\in F$ and $i=1,\ldots,r$ we have $\pi
		^\infty_*(W_i(f)) = Y_i^\infty (\pi^\infty(f))$;
		\item[(ii)] the dilations of $F$ and $M^\infty$ commute with the
		projection: namely, for any $\lambda>0$ we have
		\[
		\pi^\infty \circ \widehat\delta_\lambda = \delta_\lambda \circ
		\pi^\infty.
		\]
	\end{itemize}             
\end{proposition}

\begin{proof} (i) Using the
	action property
	\eqref{action}, we find
	\[ 
	\begin{split}
	\pi_*^\infty(W_i(f))  & =
	\left.\frac{d}{dt} \pi^\infty\big( f  \exp(tW_i )\big)\right|_{t=0}
	=	
	\left.\frac{d}{dt}  0\cdot\big( f  \exp(tW_i
	)\big)\right|_{t=0}
	\\
	& =\left.\frac{d}{dt}\pi^\infty(f) \cdot  
	\exp(tW_i)\right|_{t=0}
	=\psi(W_i)(\pi^\infty(f))
	=Y_i^\infty(\pi^\infty(f)).
	\end{split}
	\]
	
	(ii) Let $\lambda >0$ and $x\in M^\infty $.   By \eqref{pipox}, 
	for
	any
	$W\in\mathfrak f_1$ we have
	\begin{equation} 
	\label{ST}
	\delta_\lambda(x)\cdot\exp(\lambda W)=\exp(\lambda
		\psi(W))(\delta_\lambda(x))
	=\delta_\lambda\big(\exp(\psi(W))(x)\big)=\delta_\lambda(x\cdot\exp(W)).
	\end{equation}
	We deduce that the claim holds for  any $f = \exp(W)$ with $W\in\mathfrak f_1$, because 
	\[
	\pi^\infty(\widehat\delta_\lambda (f))  =
	\pi^\infty(\exp(\lambda W)) =\delta_\lambda(0)\cdot\exp(\lambda W)
	=\delta_\lambda(0\cdot\exp(W))=\delta_\lambda(\pi^\infty (f)).
	\]
	
	By Proposition \ref{generation}, any $f\in F$ is of the form $f = f_1 f_2\ldots f_k$ with each $f_i\in
	\exp(\mathfrak f_1)$. Assume by	induction that the claim holds
	for $\widehat f = f_1 f_2\ldots f_{k-1}$. By \eqref{ST}, letting $f_k=\exp(W)$
	we have
	\[
	\begin{split}
	\pi^\infty(\widehat\delta_\lambda (f)) & = \pi^\infty(\widehat\delta_\lambda(\widehat f) \exp(\lambda W))
	=\pi^\infty( \widehat\delta_\lambda(\widehat f)) \cdot   \exp(\lambda W)
	\\
	&
	= \delta_\lambda (\pi^\infty( \widehat f)) \cdot   \exp(\lambda W)
	= \delta_\lambda\big(  \pi^\infty( \widehat f) \cdot   \exp( W) \big)
	= \delta_\lambda\big(  \pi^\infty(   f) \big). \qedhere
	\end{split}
	\] 
\end{proof}

Let $\kappa:I\to M^\infty$ be a horizontal curve in $(M^\infty, \mathcal X^\infty)$,
with   control $h\in L^\infty(I,\R^r)$.
A horizontal curve $\bar\kappa:I\to F$ such that
\[
\kappa=\pi^\infty\circ\bar\kappa\qquad\text{and}\qquad\dot {\bar  \kappa }(t) = \sum_{i=1}^r h_i(t) W_i(\bar\kappa(t))\quad\text{for
	a.e.~$t\in I$}
\]
is called a {\em lift} of $\kappa$ to $(F,\mathcal W)$. 

\begin{proposition}\label{pippo}
	Let $(F,\mathcal W) $ be the  lifting of $(M^\infty,\mathcal X^\infty)$
	with projection
	$\pi^\infty: F \to M^\infty$.
	Then the following facts hold:
	\begin{itemize}
		\item[(i)]  If $\kappa$ is length minimizing in
		$(M^\infty,\mathcal X^\infty)$,
		then any horizontal lift $\bar \kappa$ of
		$\kappa$ is length minimizing in $(F,\mathcal W)$.
		
		\item[(ii)] If $\bar\kappa$ is a horizontal (half-)line in $F$,
		then
		$\pi^\infty\circ\bar\kappa$ is a horizontal (half-)line in
		$(M^\infty,\mathcal
		X^\infty)$. 
	\end{itemize}
\end{proposition}

\begin{proof}
	Claim (i) follows from   $L(\bar\kappa) = L(\kappa)$ and from the inequality $L(\bar{\kappa}')=L(\kappa')\ge L(\kappa)$, whenever $\bar{\kappa}'$ is horizontal with the same endpoints as $\bar\kappa$ and $\kappa'=\pi^\infty\circ\bar{\kappa}'$. We now turn to Claim (ii). 
	Let $\bar \kappa(t) = \exp(t W)$ for some $W\in \mathfrak f_1$. 
	The projection $\pi^\infty\circ\bar \kappa$ is horizontal by part (i) of  
	Proposition \ref{hel}. The thesis follows from characterization (i) for horizontal lines, contained in Proposition \ref{linechar}.
	
\end{proof}

\end{document}